\documentclass{article}
\usepackage{amsmath}
\usepackage{amssymb}
\usepackage{amsthm}

\usepackage{mathrsfs}

\usepackage{graphicx}

\usepackage{xypic}

\usepackage{times}

\def\A{{\mathcal A}}

\def\C{{\mathcal C}}

\def\curlyI{{\mathscr I}}

\def\L{{\mathscr L}}

\def\T{{\mathcal T}}
\def\U{{\mathcal U}}

\def\ra{\rightarrow}

\def\varep{\varepsilon}

\def\<{\langle}
\def\>{\rangle}
\def\leq{\leqslant}
\def\geq{\geqslant}

\def\fim{{\it FIM}}

\def\id{\operatorname{id}}
\def\Rat{\operatorname{Rat}}

\def\stab{\operatorname{stab}}

\def\what{\widehat}
\def\ol{\overline}

\def\cliss{closed inverse subsemigroup}
\def\clisss{closed inverse subsemigroups}
\def\clism{closed inverse submonoid}
\def\clisms{closed inverse submonoids}

\newcommand{\upset}[1]{(#1)^{\uparrow}}
\newcommand{\nobraupset}[1]{#1^{\uparrow}}
\numberwithin{equation}{section}

\title{\Large Inverse subsemigroups of finite index in finitely generated inverse semigroups}

\author{\large Amal AlAli and N.D. Gilbert
 \\
School of Mathematical and Computer Sciences\\
and the Maxwell Institute for the Mathematical Sciences,\\
Heriot-Watt University, Edinburgh EH14 4AS, U.K.\\
asa80@hw.ac.uk, N.D.Gilbert@hw.ac.uk}
\date{}
\newtheorem{theorem}{Theorem}[section]
\newtheorem{prop}[theorem]{Proposition}
\newtheorem{lemma}[theorem]{Lemma}
\newtheorem{cor}[theorem]{Corollary}

\theoremstyle{definition}
\newtheorem{remark}[theorem]{Remark}
\newtheorem{example}[theorem]{Example}

\begin{document}

\thispagestyle{empty}
\maketitle
%\begin{abstract}
%\end{abstract}
\section{Introduction}
A generalisation of the concept of \emph{coset}, from groups to inverse semigroups, was proposed by Schein in \cite{Sch}.
There are three essential ingredients to this generalisation: firstly, cosets are only to be defined for an inverse subsemigroup $L$ of an
inverse semigroup $S$ that is \emph{closed} in the natural partial order on $S$; secondly, an element $s \in S$ will only determine a coset if
$ss^{-1} \in L$; and thirdly, the coset is finally obtained by taking the closure (again with respect to the natural partial order on
$S$) of the subset $Ls$.  The details of this  construction are presented in Section \ref{intro_cosets}.  In fact, Schein takes as his
starting point a characterization of cosets in groups due to Baer \cite{Baer} (see also \cite{Cert}): a subset $C$ of a group $G$ is a coset of some subgroup $H$ of $G$ if and only if $C$ is closed under the ternary operation $(a,b,c) \mapsto ab^{-1}c$ on $G$.  

A \cliss\ $L$ of an inverse semigroup $S$ has \emph{finite index} if and only if there are only finitely many distinct cosets of
$L$ in $S$.  In contrast to the situation in group theory, finite index can arise because of the relative paucity of possible 
coset representatives satisfying $ss^{-1} \in L$.
For example, in the free inverse monoid $\fim(a,b)$, the inverse subsemigroup $\fim(a)$ has finite index.  This fact is a
consequence of a remarkable theorem of Margolis and Meakin, characterising the \clisms\ in a free inverse monoid $\fim(X)$ with
$X$ finite:

\begin{theorem}{\cite[Theorem 3.7]{MarMea}}
\label{MMthm}
Let $X$ be a finite set, and let $L$ be a \clism\ of the free inverse monoid $\fim(X)$.  Then the following conditions are equivalent:
\begin{enumerate}
\item $L$ is recognised by a finite inverse automaton,
\item $L$ has finite index in $\fim(X)$,
\item $L$ is a recognisable subset of $\fim(X)$,
\item $L$ is a rational subset of $\fim(X)$,
\item $L$ is finitely generated as a \clism\ of $\fim(X)$.
\end{enumerate}
\end{theorem}

Condition (e) of Theorem \ref{MMthm} asserts the existence of a finite set $Y \subset \fim(X)$ such that $L$ is equal to the \clism\ generated by
$Y$.  The original statement of the  theorem in \cite{MarMea} includes an extra condition related to immersions of finite graphs, which we have omitted. 

Our aims in the present paper are to present some basic facts about \clisss\ of finite index, and to study the relationships between the
conditions given in Theorem \ref{MMthm} when $\fim(X)$ is replaced by an arbitrary inverse semigroup or monoid.

In Section \ref{intro_cosets} we give an introduction to the concept of cosets in inverse semigroups, and the action of an inverse 
semigroup on a set of cosets, based closely on the ideas of Schein \cite{Sch2, Sch}.  We establish an \emph{index formula},
relating the indices $[S:K], [S:H]$ and $[H:K]$ for closed inverse subsemigroups $H,K$ of an inverse semigroup $S$ with
$E(S) \subseteq K \subseteq H$ in Theorem \ref{indexformula}, and an analogue of M. Hall's Theorem for groups,
that in a free group of finite rank, there are only finitely many subgroups of a fixed finite index, in Theorem \ref{Hall's theorem}.

Our work based on Theorem \ref{MMthm} occupies section \ref{fg_and_fi}, and is summarised in Theorem \ref{gen_MMthm}.
We show that, in an arbitrary finitey generated inverse monoid $M$, a closed inverse submonoid has finite index if and only if it is recognisable, in which case it is rational and finitely generated as a closed inverse submonoid, but that finite generation (as a 
closed inverse submonoid) is a strictly weaker property.  This is not a surprise, since any inverse semigroup with a zero is
finitely generated in the closed sense.

The authors thank Mark Lawson and Rick Thomas for very helpful comments, and in
particular for their shrewd scrutiny of Lemma \ref{star-height}.

\section{Cosets of closed inverse subsemigroups}
\label{intro_cosets}
Let $S$ be an inverse semigroup with semilattice of idempotents $E(S)$.  Recall that the \emph{natural partial order} on $S$ is defined by
\[ s \leq t \Longleftrightarrow \text{there exists} \; e \in E(S) \; \text{such that} \; s=et \,.\]
A subset $A \subseteq S$ is \emph{closed} if, whenever $a \in A$ and $a \leq s$, then $s \in A$.  The closure $\nobraupset{B}$ of a subset
$B \subseteq S$ is defined as
\[ \nobraupset{B} = \{ s \in S : s \geq b \; \text{for some} \; b \in B \}\,.\]
A subset $L$ of $S$ is \emph{full} if $E(S) \subseteq L$.

An {\em atlas} in $S$ is a subset $A \subseteq S$ such that $AA^{-1}A \subseteq A$:
that is, $A$ is closed under the {\em heap} ternary operation $\< a,b,c \> = ab^{-1}c$ (see \cite{Baer}).  Since, for all $a \in A$ we have
$\<a,a,a\>=a$, we see that $A$ is an atlas if and only if $AA^{-1}A=A$.
%
%\begin{lemma}
%Let $S$ be an inverse semigroup.
%\begin{itemize}
%\item If $A$ is an atlas in $S$ then $A^{-1}A$ and $AA^{-1}$ are inverse subsemigroups of $S$,
%\item If $A$ and $B$ are atlases in $S$ then so is their intersection $A \cap B$,
%\item Every principal ideal in $S$ is an atlas,
%\item Every local subsemigroup in $S$ is an atlas.
%\end{itemize}
%\end{lemma}
%
A {\em coset} $C$ in $S$ is a closed atlas: that is, $C$ is both upwards closed in the natural partial order on $S$ and
is closed under the heap operation $\< \dotsm \>$.

Let $X$ be a set and $\curlyI(X)$ its symmetric inverse monoid.  Let $\rho : S \ra \curlyI(X)$ be a faithful representation of $S$
on $X$, and write $x(s \rho)$ as $x \lhd s$. 

The principal characterisations of cosets that we need are due to Schein:

\begin{theorem}{\cite[Theorem 3.]{Sch}}
\label{cosets}
Let $C$ be a non-empty subset of an inverse semigroup $S$.  Then the following are equivalent:
\begin{enumerate}
\item $C$ is a coset,
\item there exists a closed inverse subsemigroup $L$ of $S$ such that, for all $s \in C$, we have $ss^{-1} \in L$ and
$C=\upset{Ls}$.
\item there exists a closed inverse subsemigroup $K$ of $S$ such that, for all $s \in C$, we have $s^{-1}s \in K$ and
$C=\upset{sK}$.
\end{enumerate}
\end{theorem}

\begin{proof}
(a) $\Longrightarrow$ (b):  Let $Q=CC^{-1} = \{ ab^{-1} : a,b \in C \}$.  Then $Q$ is an inverse subsemigroup of $S$,
since, for all $a,b,c,d \in C$ we have
\begin{itemize}
 \item $(ab^{-1})(cd^{-1})=(ab^{-1}c)d^{-1} = \< a,b,c \>d^{-1} \in Q$,
\item $(ab^{-1})^{-1} = ba^{-1} \in Q$.
\end{itemize}
Set $L=\nobraupset{Q}$: then $L$ is a closed inverse subsemigroup of $S$.  Let $s \in C$.  Obviously $ss^{-1} \in Q \subseteq L$.
Moreover, given any $c \in C$ we have
$c \geq c(s^{-1}s) = (cs^{-1})s \in Qs \subseteq Ls$, so that $C \subseteq \upset{Ls}$.  Conversely, if $x \in \upset{Ls}$,
we have $x \geq us$ for some $u \in L$, with $u \geq ab^{-1}$ for some $a,b \in C$.  Hence
$x \geq us \geq ab^{-1}c = \< a,b,c \> \in C$.  Since $C$ is closed, $x \in C$ and therefore $\upset{Ls} \subseteq C$.

(b) $\Longrightarrow$ (a): The subset $\upset{Ls}$ is a coset, since it is closed by definition, and if $h_i \in \upset{Ls}$
we have $h_i \geq t_is$ for some $t_i \in L$.  Then
$\< h_1,h_2,h_3 \> = h_1h_2^{-1}h_3 \geq t_1ss^{-1}t_2^{-1}t_3s \in Ls$ since $ss^{-1} \in L$. It follows that
$\upset{Ls}$ is closed under the heap operation $\< \dotsm \>$.

For (a) $\Longleftrightarrow$ (c): we proceed in the same way, with $K = \upset{C^{-1}C}$.
\end{proof}

For the rest of this paper, all cosets will be \emph{right} cosets, of the form $\upset{Ls}$.

\begin{prop}{\cite[Proposition 5.]{Sch}}
\label{cosets_idpt}
A coset $C$ that contains an idempotent $e \in E(S)$  is an inverse subsemigroup of $S$, and in this case $C = \upset{CC^{-1}}$.
\end{prop}

\begin{proof} If $a,b \in C$ then $ab \geq aeb = \< a,e,b \> \in C$ and since $C$ is closed, we have $ab \in C$.
Furthermore, $a^{-1} \geq ea^{-1}e = \< e,a,e \> \in C$ and so $a^{-1} \in C$.  Hence $C$ is an inverse subsemigroup.

Now $ab^{-1} \in CC^{-1}$ and $ab^{-1} \geq ab^{-1}e = \< a,b,e \> \in C$.  Since $C$ is closed we have
$\upset{CC^{-1}} \subseteq C$.  But if $x \in C$ then $x \geq xe \in CC^{-1}$ and so $x \in \upset{CC^{-1}}$.
Therefore $C = \upset{CC^{-1}}$.
\end{proof} 

Now if $L$ is a closed inverse subsemigroup of $S$, a coset of $L$ is a subset of the form $\upset{Ls}$ where
$ss^{-1} \in L$.  Suppose that $C$ is such a coset: then Theorem \ref{cosets} associates to $C$ the closed inverse
subsemigroup $\upset{CC^{-1}}$. 

\begin{prop}{\cite[Proposition 6.]{Sch}}
\label{cosetsofL}
Let $L$ be a closed inverse subsemigroup of $S$.
\begin{enumerate}
\item Suppose that $C$ is a coset of $L$.  Then $\upset{CC^{-1}}=L$.
\item If $t \in C$ then $tt^{-1} \in L$ and $C = \upset{Lt}$.  Hence two cosets of $L$ are either disjoint or they coincide.
\item Two elements $a,b \in S$ belong to the same coset $C$ of $L$ if and only if $ab^{-1} \in L$.
\end{enumerate}
\end{prop}

\begin{proof}
(a) If $c_i \in C$ ($i=1,2$) then there exists $l_i \in L$ such that $c_i \geq l_is$.
Hence $c_1c_2^{-1} \geq l_1ss^{-1}l_2^{-1} \in L$, and so $CC^{-1} \subseteq L$.  Since $L$ is closed,
$\upset{CC^{-1}} \subseteq L$.  On the other hand, for any $l \in L$ we have
$l=ll^{-1}l \geq lss^{-1}l^{-1}l = (ls)(l^{-1}ls)^{-1} \in CC^{-1}$ and so $L \subseteq \upset{CC^{-1}}$.

(b) If $C = \upset{Ls}$ and $t \in C$ then, for some $l \in L$ we have $t \geq ls$.  Then 
$tt^{-1} \geq lss^{-1}l^{-1} \in L$, and since $L$ is closed, $tt^{-1} \in L$.  Moreover, if
$u \in \upset{Lt}$ then for some $k \in L$ we have $u \geq kt \geq kls$ and so $u \in \upset{Ls}$.  Hence if
$t \in \upset{Ls}$ then $\upset{Lt} \subseteq \upset{Ls}$.   Now
$ls = (ls)(ls)^{-1}t = lss^{-1}l^{-1}t$ and so $l^{-1}ls = l^{-1}lss^{-1}l^{-1}t = ss^{-1}l^{-1}t \in Lt$.
Since $s \geq l^{-1}ls$, we deduce that $s \in \upset{Lt}$.  Hence $\upset{Ls} \subseteq \upset{Lt}$.

(c)  Suppose that $a,b \in \upset{Ls}$.  Then for some $k,l \in L$ we have $a \geq ks$ and $b \geq ls$: hence
$ab^{-1} \geq kss^{-1}l^{-1} \in L$ and so $ab^{-1} \in L$.  On the other hand, suppose that $ab^{-1} \in L$.
Then $aa^{-1} \geq a(b^{-1}b)a^{-1} = (ab^{-1})(ab^{-1})^{-1} \in L$, and similarly $bb^{-1} \in L$. We note that
$a = (aa^{-1})a \in La$ and similarly $b \in Lb$. Then
$a \geq a(b^{-1}b) = (ab^{-1})b$ and so $a \in \upset{Lb}$.  As in part (b) we deduce that $\upset{La} \subset \upset{Lb}$.
By symmetry $\upset{La} = \upset{Lb}$ and this coset contains $a$ and $b$.
\end{proof}

\begin{example}
\label{Eclosed}
 Let $E$ be the semilattice of idempotents of $S$.  The property that $E$ is closed is exactly the property that $S$ is
\emph{$E$--unitary}.  In this case, for any $s \in S$, we have
\begin{align*}
\upset{Es} &= \{ t \in S : t \geq es \; \text{for some} \; e \in E \} \\
&= \{ t \in S : t \geq u \leq s \; \text{for some} \; u \in S \} \\
&= \{ t \in S : s,t \; \text{have a lower bound in} \; S \}.
\end{align*}
We see that $\upset{Es}$ is precisely the $\sigma$--class of $s$, where $\sigma$ is the minimum group congruence on $S$,
see \cite[Section 2.4]{LwBook}.  Hence every element $t \in S$ lies in a coset of $E$, and the set of cosets is in
one-to-one correspondence with the maximum group image $\what{S}$ of $S$.
\end{example}

\begin{remark}
Let  $L$ be a closed inverse subsemigroup of an inverse semigroup $S$. Then the union $U$, of all the cosets of $L$ is a subset of $S$ but need not be all of $S$, and is not always a subsemigroup of $S$.
\end{remark}

We illustrate this remark in the following example.

\begin{example}
Fix a set $X$ and recall that the \emph{Brandt semigroup} $B_X$ is defined as follows.  As a set, we have
\[ B_X = \{ (x,y) : x,y \in X \} \cup \{ 0 \} \]
with
\[ (u,v)(x,y) = 
\begin{cases}
(u,y) & \text{if} \; v=x \\
0 & \text{if} \; v \ne x
\end{cases} \]
and $0(x,y) = 0 = (x,y)0$.
The idempotents of $B_X$ are the elements $(x,x)$ for $x \in X$ and $0$.  Hence $0 \leq (x,y)$ for all $x,y \in X$ and
$(u,v) \leq (x,y)$ if and only if $(u,v)=(x,y)$.  If a closed inverse semigroup $L$ of $B_X$ contains $(x,y)$ with $x \ne y$
then $(x,y)(x,y)=0 \in L$ and so $L=B_X$.  Therefore the only proper closed inverse subsemigroups are the subsemigroups
$E_x = \{ (x,x) \}$ for $x \in X$.  An element $(x,y) \in B_X$ then determines the coset
\[ \upset{E_x(x,y)} = \{ (x,y) \} \,.\]
Hence there are $|X|$ distinct cosets of $E_x$ and their union is
\[ U = \{ (x,y) : y \in X \} \,. \]
\end{example}

\begin{prop}
\label{cosetunion}
\leavevmode
\begin{enumerate}
\item Let $L$ be a closed inverse subsemigroup of an inverse semigroup $S$ and let $U$ be the union of all the cosets of $L$ in $S$. Then
$   U= \left\{ s\in S : ss^{-1} \in L \right\} $
and therefore $U=S$ if and only if $L$ is full.
\item $U$ is a closed inverse subsemigroup of $S$ if and only if whenever $e \in E(L)$ and $s \in U$ 
then $ses^{-1} \in U$ and if,
whenever $s \in S$ with 
$ss^{-1}\in L$, then $s^{-1}s \in L \,.$
\end{enumerate}
\end{prop}

\begin{proof}
(a) The coset $\upset{Lu}$ containing $u \in S$ exists if and only if $uu^{-1} \in L$.

(b) Suppose that $L$ satisfies the given conditions.  If $s,t \in U$ then $ss^{-1}, tt^{-1} \in L$ and
\[ (st)(st)^{-1} = s(tt^{-1})s^{-1} \in U \]
which implies that $st \in U$, and $s^{-1}s \in L$ which implies that $s^{-1} \in U$.  Hence $U$ is an inverse subsemigroup.
Conversely, if $U$ is an inverse subsemigroup and $s \in U$, then $s^{-1} \in U$ which implies that $s^{-1}s \in L$,
and if $s \in U$ and $e \in E(L)$ then $e \in U$ and so $se \in U$ which implies that $(se)(se)^{-1} = ses^{-1} \in U$.  
\end{proof}

\subsection{The index formula}
The \textit{index} of the closed inverse subsemigroup $L$ in an inverse semigroup $S$ is the cardinality of the set of right cosets of $L$, and is written $[S:L]$.   Note that the mapping $\upset{Ls} \rightarrow \upset{s^{-1}L}$ is a bijection from the set of right cosets to the set of left cosets.  A \emph{transversal} to $L$ in $S$ is a choice of one element from each right coset of $L$.   For a transversal $\T$,  we have the union
               $$U = \bigcup_{t \in \T}\,\upset{Lt}, $$
as in Proposition \ref{cosetunion}, and each element $u \in U$ satisfies $u \geq ht$ for some $h \in L,\, t \in \T.$

\begin{theorem}\label{indexformula}
Let $S$ be an inverse semigroup and let $H$ and $K$ be two closed inverse subsemigroups of $S$ with $K$ of finite index in $H$ and $H$ of finite index in $S$, and with $K \subseteq H$ and $K$ full in $S.$
Then $K$ has finite index in $S$ and
$$ [ S:K ] = [\, S:H ]  [\, H:K ]. $$
\end{theorem}

\begin{proof}   
Since $K$ is full in $S$, then so is $H$ and for transversals $\T, \U$ we have
$$S = \bigcup_{t \in \T} \upset{Ht} \quad \text{and} \quad H = \bigcup_{u \in \U} \upset{Ku}.$$
Therefore
$$S = \{ s \in S : s \geq ht \quad \textit{for some} \; t \in \T, h \in H \},$$ and
$$H = \{ s \in S : s \geq ku \quad \textit{for some} \; u \in \U,   k \in K \}.$$
Now if $s \geq ht$ and $h \geq ku$ then $s \geq kut $. Then 
$s\in\,\upset{Kut}$ and $\upset{Kut}$ is a coset of $K$ in $S$, since $K$ is full in $S$ and therefore $ (ut)\,(ut)^{-1}\,\in\,K.$\\
Hence
$$ S = \bigcup_{\substack{u \in U  \\    t \in \T}} \, \upset{Kut} \,.$$

It remains to show that all the cosets $\upset{Kut}$ are distinct. 
Suppose that $\upset {Ku't'} = \upset{Kut}$.
Then by part (c) Proposition \ref{cosetsofL}, $u't't^{-1}u^{-1}\in K$ and so $u't't^{-1}u^{-1} \in H$.
Since $u,u' \in H$ we have
$(u')^{-1}u't't^{-1}u^{-1}u \in H$
and since
$t't^{-1} \geq  (u')^{-1}u't' t^{-1}u^{-1}u \in H$
and $H$ is closed, then $t't^{-1} \in H .$ 
This implies that 
$ \upset {Ht'}= \upset {Ht}$ and it follows that $t'=t$.

Now 
$u't't^{-1}u^{-1} \in K$. Since $t' = t$, then $t't^{-1} \in E(S)$ and so
$u'u^{-1} \geq u't't^{-1}u^{-1}$.
But $K$ is closed, so $u'u^{-1} \in K$ and $\upset{Ku'} = \upset {Ku}$.  Hence $u'=u$.
Consequently, all the cosets $\upset{Kut}$ are distinct.
\end{proof}

Recall from Example \ref{Eclosed} that the property that $E(S)$ is closed is expressed by saying that $S$ is
$E$--unitary and that  in this case, the set of cosets of $E(S)$ is in
one-to-one correspondence with the maximum group image $\what{S}$ of $S$.  

\begin{prop}
\label{fullfindex}
Let $S$ be an $E$--unitary inverse semigroup.  Then:
\begin{enumerate}
\item[(a)] $E(S)$ has finite index if and only if the maximal group image $\what{S}$ is finite, and $[ S: E ] = |\what{S}|$,
\item[(b)] if $E(S)$ has finite index in $S$ then, for any closed, full, inverse subsemigroup $L$ of $S$ we have
\[ [ S:L ] = |\what{S}| \, / \,|\what{L}| \,\]
\end{enumerate}
\end{prop}

\begin{proof}
Part (a) follows from our previous discussion.  For part (b) we have $E(S) \subseteq L \subseteq S$ and so if the index $[S:E]$ is finite 
then so are $[S:L]$ and $[L:E]$ with $[S:E]=[S:L][L:E]$.  But now $[S:E]=|\what{S}|$ and $[L:E]=|\what{L}|$.  
\end{proof}

The index formula in Theorem \ref{indexformula} can still be valid when $K$ is not full in $S$ as we show in the following Example.  

\begin{example}
\label{clism_of_I3}
We work in the  symmetric inverse monoid  $\curlyI_3 = \curlyI \left(\{1,2,3\}\right)$, and take  $L=stab(1) =\{\sigma \in \curlyI (X):\, 1\sigma =1 \}$,
which is a closed inverse subsemigroup of $ \curlyI (X)$ with $7$ elements.  There are $3$ cosets of $L$ in $\curlyI_3$, namely
$$  C_1 = \{\sigma \in \curlyI_3 : 1\sigma =1 \} = L,$$ 
$$ C_2 = \{\sigma \in \curlyI_3 : 1\sigma =2 \},$$ 
$$ C_3 = \{\sigma \in \curlyI_3 : 1\sigma =3 \},$$ 
and so  $[\curlyI_3:L]=3\,.$ 

Now take 
\[ K = \,\left\{
\begin{pmatrix} 
 1 & 2  &  3 \\
  1   & 2  & 3\\
\end{pmatrix},
\begin{pmatrix} 
 1 & 2  &  3 \\
  1  & 3  & 2\\
\end{pmatrix}\right\}\,.\]
Then $K$ is a closed inverse subsemigroup of $L.$ The domain of each $\sigma$ in $K$ is $\{1,2,3\}.$ and so the only coset
representatives for $K$ in $L$ are the permutations 
\[ \id= \begin{pmatrix} 
 1 & 2  &  3 \\
  1   & 2  & 3\\
\end{pmatrix} \; \text{and}\:\: \;
\sigma = \begin{pmatrix} 
 1 & 2  &  3 \\
  1  & 3  & 2\\
\end{pmatrix}.\]
But these are elements of $K$ and so $\nobraupset{K}=\upset{K \sigma}=K$ and there is just one coset.  Hence $[ L:K ]=1$.

Now, we calculate the cosets of $K$ in  $ \curlyI_3.$  Each permutation in $\curlyI_3$ is a possible coset representative
for $K$ and these produce three distinct cosets by Proposition \ref{cosetsofL}(3).  Hence $[\curlyI_3 :K ]=3$
and in this example
\[ [\curlyI_3:K ] = [\curlyI_3:L ] [ L:K ] \,.\]
\end{example}

The index of a \cliss\ $L$ of an inverse semigroup $S$ depends on the availability of coset representatives to make cosets, and so on the idempotents of $S$ contained in $L$.  In particular, we can have $K \subset L$ but $[S:K| < [S:L]$ as the following Example illustrates.  

\begin{example}
\label{index_in_fimxy}
Consider the free inverse monoid $\fim(x,y)$ and the closed inverse submonoids $K = \nobraupset{ \< x^2 \>}$ and
$H = \nobraupset{\< x^2,y^2 \>}$.  As in \cite{MarMea, Munn}, we represent elements of $\fim(x,y)$ by Munn
trees $(P,w)$ in the Cayley graph $\Gamma(F(x,y),\{x,y\})$ where $F(x,y)$ is the free group on $x,y$.

Consider a coset $\nobraupset{K(P,w)}$.  For this to exist, the idempotent $(P,1)=(P,w)(P,w)^{-1}$ must be in $K$ and so as a 
subtree of $\Gamma$, $P$ can only involve vertices in $F(x)$ and edges between them.  Since $w \in P$ we must have
$w \in F(x)$.  It is then easy to see that there are only two cosets, $K$ and $\upset{Kx}$ and so $[\fim(x,y):K]=2$.
Similarly, $[H:K]=1$.

Now consider a coset $\nobraupset{H(P,w)}$.  Now $(P,1) \in H$ and so $P$ must be contained in the subtree of $\Gamma$ spanned by the
vertices of the subgroup $\< x^2,y^2 \> \subset F(x,y)$, and $w \in P$.  If $w \in \< x^2,y^2 \>$
then $(P,w) \in H$ and $\nobraupset{H(P,w)}=H$.  Otherwise, $w = ux$ or $w=uy$ with $u \in \< x^2,y^2 \>$ and it
follows that there are three cosets of $H$ in $\fim(x,y)$, namely $H, \upset{Hx}$ and $\upset{Hy}$.  Therefore
\[ [\fim(x,y):H] = 3 > [\fim(x,y):K] =2 \,.\]
These calculations also follow from results of Margolis and Meakin, see \cite[Lemma 3.2]{MarMea}.

We note that the index formula fails to hold.  This does not contradict Theorem \ref{indexformula}, since $K$ is not full in 
$\fim(x,y).$
\end{example}

\subsection{Coset actions}
Our next aim is to derive an analogue of Marshall Hall's Theorem (see \cite{MHallJr} and \cite{BaumBook}) that, in a free group of finite rank, there are only finitely many subgroups of a fixed finite index.  We first record some preliminary results on actions on cosets:
these results are due to Schein \cite{Sch2} and are presented in \cite[Section 5.8]{HoBook}.  We give them here  for the reader's convenience.

\begin{lemma}
\label{rt_upset}
For any subset $A$ of $S$ and for any $u \in S$ we have $\upset{Au} = \upset{\nobraupset{A}u}$.  
\end{lemma}

\begin{proof}
Since $A \subseteq \upset{A}$ it is clear that $\upset{Au} \subseteq \upset{\upset{A}u}$.  On the other hand, if
$x \in \upset{\upset{A}u}$ then for some $a \in A$ we have $y \in S$ with $y \geq a$ and $x \geq yu$.  But then
$x \geq au$ and so $x \in \upset{Au}$.
\end{proof}

Let $L$ be a closed inverse subsemigroup of $S$, and let $s \in S$ with $ss^{-1} \in L$ with $C=\upset{Ls}$
Now suppose that $u \in S$ and that $\upset{Cuu^{-1}}=C$.  Then we define $C \lhd u = \upset{Cu}$.

\begin{lemma}
\label{eq_action_cond}
The condition $\upset{Cuu^{-1}}=C$ for $C \lhd u$ to be defined is equivalent to the condition that
$suu^{-1}s^{-1} \in L$.
\end{lemma}

\begin{proof}
Since for any $c \in C$ we have $c \geq cuu^{-1}$ it is clear that $C \subseteq \upset{Cuu^{-1}}$.
Now suppose that $suu^{-1}s^{-1} \in L$ and that $x \in \upset{Cuu^{-1}} = \upset{Lsuu^{-1}}$ (by
Lemma \ref{rt_upset}).  Hence there exists $y \in L$ with 
$$x \geq ysuu^{-1} = yss^{-1}suu^{-1} = ysuu^{-1}s^{-1}s \in Ls.$$
Therefore $x \in \upset{Ls} = C$ and so $\upset{Cuu^{-1}}=C$.  On the other hand, if $\upset{Cuu^{-1}}=C$,
then $\upset{Lsuu^{-1}} \subseteq \upset{Ls}$.  Since $ss^{-1} \in L$ we have
$ss^{-1}suu^{-1} = suu^{-1} \in \upset{Ls}$ and so there exists $y \in L$ with $suu^{-1} \geq ys$.  But then
$suu^{-1}s^{-1} \geq yss^{-1} \in L$ and since $L$ is closed we deduce that $suu^{-1}s^{-1} \in L$.
\end{proof}

It follows from Lemma \ref{eq_action_cond} that the condition $suu^{-1}s^{-1} \in L$ does not depend on the
choice of coset representative $s$.  This is easy to see directly.  If $\upset{Ls} = \upset{Lt}$ then, by 
part (c) of Proposition \ref{cosetsofL} we have $st^{-1} \in L$.  Then
$$tuu^{-1}t^{-1} \geq ts^{-1}suu^{-1}s^{-1}st^{-1} = (st^{-1})^{-1}(suu^{-1}s^{-1})(st^{-1}) \in L$$
and since $L$ is closed, $tuu^{-1}t^{-1} \in L$.

\begin{prop}
\label{transaction}
If $u \in S$ and $\upset{Cuu^{-1}}=C$ then $\upset{Cu} = \upset{Lsu}$ and the rule $C \lhd u = \upset{Cu}$
defines a transitive action of $S$ by partial bijections on the cosets of $L$.
\end{prop}

\begin{proof}
Since $C=\upset{Ls}$, Lemma \ref{rt_upset} implies that $\upset{Cu} = \upset{Lsu}$.  To check that $\upset{Lsu}$ is
a coset of $L$, we need to verify that $(su)(su)^{-1} \in L$.  But $(su)(su)^{-1} = suu^{-1}s^{-1}$ and so this follows 
from Lemma \ref{eq_action_cond}. Moreover, $(Cu) \lhd u^{-1}$ is defined and equal to
$\upset{Cuu^{-1}}=C$, so that the action of $u$ is a partial bijection.

It remains to show that for any $s,t \in S$, the action of $st$ is the same as the action of $s$ followed by the action of $t$ whenever
these are defined.  Now the outcome of the actions are certainly the same:  for a coset $C$, we have $C \lhd (st) = \upset{Cst}$
and 
$$(C \lhd s) \lhd t = \upset{Cs} \lhd t = \upset{\upset{Cs}t} =  \upset{Cst}$$
by Lemma \ref{rt_upset}.

The conditions for $C \lhd (st)$, $C \lhd s$, $(C \lhd s) \lhd t$ to be defined are, respectively:
\begin{align}
& \upset{Cstt^{-1}s^{-1}} = C \,, \label{defCst}\\
& \upset{Css^{-1}}=C \,, \label{defCs}\\
& \upset{\upset{Cs}tt^{-1}} = \upset{Cstt^{-1}} = \upset{Cs}. \label{def(Cs)t}
\end{align}
Suppose that \eqref{defCst} holds.  Then
$$\upset{Css^{-1}} \subseteq \upset{Cstt^{-1}s^{-1}} = C.$$
But it is clear that $C \subseteq \upset{Css^{-1}}$, and so $\upset{Css^{-1}}=C$ and \eqref{defCs} holds.  Now 
it is again clear that $\upset{Cs} \subseteq \upset{Cstt^{-1}}$, and
$$\upset{Cstt^{-1}} \subseteq \upset{Cstt^{-1}s^{-1}s} = \upset{\upset{Cstt^{-1}s^{-1}}s} = \upset{Cs}.$$
Therefore \eqref{def(Cs)t} holds.

Now if both \eqref{defCs} and \eqref{def(Cs)t} hold we have
\begin{align*}
\upset{Cstt^{-1}s^{-1}} &= \upset{\upset{Cstt^{-1}}s^{-1}}  & \text{by Lemma \ref{rt_upset}} \\
&= \upset{\upset{Cs}s^{-1}}  & \text{by \eqref{def(Cs)t}} \\
&= \upset{Css^{-1}}  & \text{by Lemma \ref{rt_upset}} \\
&= C  & \text{by \eqref{defCs}}.
\end{align*}
and therefore \eqref{defCst} holds.

To show that the action is transitive, consider two cosets $\upset{La}$ and $\upset{Lb}$.  Then
$\upset{La} \lhd a^{-1}b$ is defined since $a(a^{-1}b)(a^{-1}b)^{-1}a^{-1}=aa^{-1}bb^{-1} \in L$,
and $\upset{La} \lhd a^{-1}b = \upset{Laa^{-1}b} = \upset{Lb}$, since again $aa^{-1}bb^{-1} \in L$.
\end{proof}

\subsection{Marshall Hall's Theorem for inverse semigroups}
\begin{theorem}
\label{Hall's theorem}
In a finitely generated inverse semigroup $S$ there are at most finitely many distinct closed inverse subsemigroups of a fixed finite index 
$d$. 
\end{theorem}

\begin{proof}
Suppose that the inverse semigroup $S$ is finitely generated and that the closed inverse subsemigroup $L$ of $S$ has exactly $d$ cosets. We aim to construct an inverse semigroup homomorphism 
   $$   \phi_{L}: S\longrightarrow \curlyI(D), $$ 
where $ \curlyI(D) $ is the symmetric inverse monoid on $D=\{1,...,d\}.$

Write the  distinct cosets of $L$ as $\upset{L c_{1}},\upset{L c_{2}},...\,,\upset{L c_{d}},$ with $c_{1},c_{2},...,c_{d}\in S,$ and with $\upset{L c_{1}}=L.$ Now take $u \in S.$ If $c_{j}\,u\,u^{-1}\,{c_{j}}^{-1}\in L,$ where $ j\in \{1,...,d\},$ then we can define an action of the element $u \in S$ on the coset $\upset{Lc_j}$ of $L$ as follows:
$$\upset{Lc_{j}} \lhd u = \upset{Lc_{j}u}.$$ 
By Proposition \ref{transaction}, $\upset{Lc_{j}u}$ is indeed a coset of $L,$ and so $\upset{Lc_{j}u}  =\upset{Lc_{k}},$ where $ k\in \{1,...,d\}.$ Then we can write $\upset{Lc_{j}} \lhd u = \upset{Lc_{k}},$ and this action of $u$ induces an action $j \lhd u = k$ of $u$ on $D$, and so we get a homomorphism
$$   \phi_{L}: S\longrightarrow \curlyI(D) . $$
We now claim that different choices of $L$ give us different homomorphisms $\phi_{L}$, or equivalently, that if
$\phi_{L}=\phi_{K}$ then $L=K$.  

By Proposition \ref{cosets_idpt}, if $x \in L$ then $L = \upset{Lxx^{-1}}$.  By Lemma \ref{eq_action_cond} $L \lhd x$ is defined and is 
 equal to $\upset{Lx}=L$.  Now suppose that $L \lhd y$ is defined and that $\upset{Ly}$ is equal to $L$.  By Lemma \ref{eq_action_cond}
 we have $\upset{Lyy^{-1}}=L$.  Hence $yy^{-1} \in L$, and $y = yy^{-1}y \in \upset{Ly}=L$.   It follows that 
$\stab(L)=L$ and in the induced action of $S$ on $D$ we have $\stab(1)=L,$ so that $L$ is determined by $\phi_{L}\,.$ 
 Therefore, the number of closed inverse subsemigroups of index $d$ is at most the number of homomorphisms $\phi : S \longrightarrow \curlyI (D)$, and
since $S$ is finitely generated, this number is finite.
 \end{proof}

\section{Finite generation and finite index}
\label{fg_and_fi}
In this section, we shall look at the properties of closed inverse submonoids of free inverse monoids considered in Theorem
\ref{MMthm}, and the relationships between these properties when we replace a free inverse monoid by an arbitrary inverse monoid.
Throughout this section, $M$ will be an inverse monoid generated by a finite subset $X$.  This means that the smallest inverse submonoid $\< X \>$
of $M$ that contains $X$ is $M$ itself: equivalently, each element of $M$ can be written as a product of elements of $X$ and inverses of elements of $X$, so if we set $A= X \cup X^{-1}$ then each element of $M$ can be written as a product of elements in $A$.
A closed inverse submonoid $L$ of $M$ is said to be \emph{finitely generated as a closed inverse submonoid} if there exists a finite subset $Y \subseteq L$
such that, for each $\ell \in L$ there exists a product $w$ of elements of $Y$ and their inverses such that $\ell \geq w$.  Equivalently, the
smallest closed inverse submonoid of $M$ that contains $Y$ is $L$.
We remark that in \cite{MarMea} the notation $\< X \>$ is used for the smallest \emph{closed} inverse submonoid of $M$ that
contains $X$.  We shall use  $\nobraupset{\< X \>}$ for this.

We will need to use some ideas from the theory of finite automata and for background information on this topic we refer to  \cite{LwautBook, Pin, SipBook}.

A \emph{deterministic finite state automaton} $\A$ (or just an \emph{automaton} in this section) consists of
\begin{itemize}
\item a finite set $S$ of \emph{states},
\item a finite \emph{input alphabet} $A$,
\item an \emph{initial state} $s_0 \in S$,
\item a partially defined \emph{transition function} $\tau : S \times A \ra S$,
\item a subset $T \subseteq S$ of \emph{final} states.
\end{itemize}

We shall write $s \lhd a$ for $\tau(s,a)$ if $\tau(s,a)$ is defined.  Given a word $w = a_1 a_2 \dotsm a_m \in A^*$ we write
$s \lhd w$ for the state $( \dots (s \lhd a_1) \lhd a_2) \lhd \dotsb ) \lhd a_m,$ that is, for the state obtained from $s$ by computing the
succesive outcomes, if all are defined,  of the transition function determined by the letters of $w$, with the empty word $\varep$ acting by $s \lhd \varep = s$ for all $s \in S$.  We normally think of
an automaton in terms of its \emph{transition diagram}, in which the states are the vertices of a directed graph
and the edge set is $S \times A$, with an edge $(s,a)$ having source $s$ and target $s \lhd a$.

%Each $a \in A$ determines a transformation $\tau_a : S \ra S$ defined by $s \tau_a = s \lhd a$.  Hence $\tau_a$ is a an element
%of the monoid $\P_S$ of all partially defined  functions $S \ra S$, and the functions $a \mapsto \tau_a$ then induces a monoid %homomorphism
%$\tau^*: A^* \ra \P_S$, whose image is called the \emph{transition monoid} of $\A$.

Let $X$ be a finite set, $X^{-1}$ a disjoint set of formal inverses of elements of $X$, and $A=X \cup X^{-1}$  An automaton 
$\A$ with input alphabet $A$ is called a \emph{dual automaton} if, whenever $s \lhd a = t$ then $t \lhd a^{-1} =s.$  A dual
automaton is called an \emph{inverse automaton} if, for each $a \in A$ the partial function $\tau(-,a): S \ra S$ is injective.  
(See \cite[Section 2.1]{LwBook}.)

A word $w \in A^*$ is \emph{accepted} or \emph{recognized} by $\A$ if $s_0 \lhd w$ is defined and $s_0 \lhd w \in T$.  The set of all words recognized by 
$\A$ is the \emph{language} of $\A$:
\[ \L(\A) = \{ w \in A^* : s_0 \lhd w \in T \}\,. \]
A language $\L$ is \emph{recognizable} if it is the language recognized by some automaton. The connection between automata and closed inverse subsemigroups of finite index is made, as in \cite{MarMea}, by the coset automaton.

Let $M$ be a finitely generated inverse monoid, generated by $X \subseteq M$, and let $L$ be a closed inverse
submonoid of $M$ of finite index.  Since $M$ is generated by $X$, there is a natural monoid homomorphism $\theta: A^* \ra M$.   The \emph{coset automaton} $\C=\C(M:L)$ is defined as follows:  
\begin{itemize}
\item the set of states is the set of cosets of $L$ in $M$, 
\item the input alphabet is $A = X \cup X^{-1}$,
\item the initial state is the coset $L$,
\item the transition function is defined by $\tau(\upset{Lt},a) = \upset{Lt(a \theta)}$,
\item the only final state is $L$.
\end{itemize}
By Lemma \ref{eq_action_cond} and Proposition \ref{transaction}, $\upset{Lt} \lhd a$ is defined if and only if
$t(a \theta)(a \theta)^{-1}t^{-1} \in L$.
The following Lemma 
occurs as \cite[Lemma 3.2]{MarMea} for the case that $M$ is the free inverse monoid $\fim(X)$.

\begin{lemma}
\label{coset_aut}
The coset automaton of $L$ in $M$ is an inverse automaton. 
The language $\L(\C(M:L))$ that it recognizes is
\[ L \theta^{-1} = \{ w \in A^* : w \theta \in L \}\]
and $\C(M:L)$ is the minimal automaton recognizing $L \theta^{-1}$.
\end{lemma}

\begin{proof}
It follows from Proposition \ref{transaction} that $\C(M:L)$ is inverse.
Suppose that $w$ is recognized by $\C(M:L)$.  Then $(w \theta)(w\theta)^{-1} = (ww^{-1})\theta \in L$
and $L\upset{w \theta}=L.$  From Proposition \ref{cosetsofL}, we deduce that $w \theta \in L$.
Conversely, suppose that $w = a_{i_1} \dotsc a_{i_m} \in A^*$ and that $s = w \theta \in L$.  For $1 \leq k \leq m$, write
$p_k = a_{i_1} \dotsc a_{i_k}$, $q_k = a_{i_{k+1}} \dotsc a_{i_m}$, so that $w = p_kq_k,$ and take $s_k = p_k\theta,$ so that $s_1 = a_{i_1}\theta$.  Then
\[ s_1 s_1^{-1}s = s_1s_1^{-1}s_1(q_2 \theta) = s_1(q_2 \theta) = w \theta =s \]
and so $s_1 s_1^{-1} \geq ss^{-1} \in L$.  Therefore $s_1 s_1^{-1} \in L$ and $L \lhd a_{i_1} = \upset{Ls_1}$ is defined.
Now suppose that for some $k$ we have that $L \lhd w_k$ is defined and is equal to $\upset{Ls_k}$.
Then 
\[ s_{k+1}\,s_{k+1}^{-1}\,s = s_{k+1}\,s_{k+1}^{-1}\,s_{k+1}(q_{k+1} \theta)
= s_{k+1}\,(q_{k+1} \theta) = w \theta = s\]
and so $s_{k+1}\,s_{k+1}^{-1} \geq ss^{-1} \in L$ and therefore $s_{k+1}s_{k+1}^{-1} \in L.$
But 
\[ s_{k+1}\,s_{k+1}^{-1} = s_k(a_{i_{k+1}}\theta)\,(a_{i_{k+1}}\theta)^{-1}s_k^{-1} \in L \,,\]
and so by Lemma \ref{eq_action_cond}, $\upset{Ls_k} \lhd a_{i_{k+1}}$ is defined and is equal to
$\upset{Ls_k(a_{i_{k+1}}\theta)} = \upset{Ls_{k+1}}$.  It follows by induction that
$L \lhd w$ is defined in $\C(M:L)$ and is equal to $\upset{Ls} = L,$ and so $w \in L(\C(M:X))$.
Now by a result of Reutenauer \cite[Lemme 1]{Reut}, a connected inverse automaton with one initial and one final state
is minimal.
\end{proof}

The set of \emph{rational} subsets of $M$ is the smallest collection that contains all the finite subsets of $M$
and is closed under finite union, product, and generation of a submonoid.  Equivalently, $R \subseteq M$ is a rational subset of $M$ if and only if there exists a recognizable subset $Z \subseteq A^*$ with $Z \theta = R$ (see \cite[section IV.1]{Pin}).

We also recall the notion of \emph{star-height} of a rational set (see \cite[Chapter III]{Ber}).
Let $M$ be a monoid. Define a sequence of subsets $\Rat_h(M),$ with  \emph{star-height} $h\geq0,$ recursively as follows:
$$\Rat_{0}(M)=\{X \subseteq M| \,X \text{\,is \,finite\,}\},$$
and $\Rat_{h+1}(M)$ consists of the finite unions of sets of the form $B_{1}B_2 \dotsm B_{m}$ where each $B_{i}$ is either a singleton or $B_{i}=C_{i}^*,$ for some $C_{i}\in \Rat_{h}(M).$  It is well known that $\Rat(M)=\bigcup_{h\geq0} \Rat_{h}(M)$.

A subset $S$ of $M$ is \emph{recognizable}
if there exists a finite monoid $N,$ a monoid homomorphism $\phi: M \ra N$, and a subset $P \subseteq N$ such that
$S = P \phi^{-1}.$   For free monoids $A^*,$ Kleene's Theorem (see for example \cite[Theorem 5.2.1]{LwautBook})
tells us that the rational and recognizable subsets coincide.  For finitely generated monoids, we have the following theorem
due to McKnight.

\begin{theorem}
\label{mcknight}
In a finitely generated monoid $M$, every recognizable subset is rational.
\end{theorem}

If $M$ is generated (as an inverse monoid) by $X$, then as above we have a monoid homomorphism $\theta: A^* \ra M$.
We say that a subset $S$ of $M$ is \emph{recognized} by an automaton $\A$ if its full inverse image $S \theta^{-1}$ in $A^*$
is recognized by $\A$.  We shall use the Myhill-Nerode Theorem \cite{Myhill, Nerode} to characterize recognizable languages.
Let $K \subseteq A^*$ be a language.
Two words $u,v \in A^*$ are \emph{indistinguishable by} $K$ if, for all $z \in A^*$, $uz \in K$ if and only if $vz \in K$.
We write $u \simeq_K v$ in this case: it is easy to check that $\simeq_K$ is an equivalence relation (indeed, a right congruence)
on $A^*$.  The we have:

\begin{theorem}[The Myhill-Nerode Theorem]
\label{nerode Thm}
A language $\L$ is recognizable  if and only if the equivalence relation $ \simeq_{\L}$ has finitely many classes.
\end{theorem}

We refer to \cite[Section 9.6]{LwautBook} for more information about, and a sketch proof of this result.

\subsection{Finite index implies finite generation}
In this section, we consider a closed inverse submonoid $L$ that has finite index in a finitely generated inverse monoid
$M$. We shall show that $L$ is finitely generated as a closed inverse submonoid.  Our proof differs from that given in \cite[Theorem 3.7]{MarMea}
for the case $M = \fim(X)$: instead we generalize the approach taken for groups in \cite[Theorem 3.1.4]{BaumBook}.  Recall that a  transversal to
$L$ in $M$ is a choice of one representative element from each coset of $L$.  We always choose the element $1_M$ from the coset $L$
itself.  For $s \in S$ we write $\ol{s}$ for the \textit{representative} of the coset that contains $s$ (if it exists), and note the following:

\begin{lemma}\label{delta}
Let $\T$ be a transversal to $L$ in $M$ and define, \,for $r \in \T$ and $s \in M$,
$\delta (r,s)=rs\,{(\overline{rs})}^{-1} \,.$
Then for all $s,t \in M$, with $ss^{-1}, stt^{-1}s^{-1} \in L$,
\begin{enumerate}
\item $\upset{Ls}= \upset{L\overline{s}}$
\item $\overline{\overline{s}t}= \overline{st}$ 
\item $s \geq \delta (1_M,s)\: \overline{s}\,\,.$
\end{enumerate}
\end{lemma}

\begin{proof}
Parts (a) and (c) are clear: part (b) is a special case of Lemma \ref{rt_upset}.
\end{proof}

\begin{theorem}
 \label{Schr's lemma}
A closed  inverse submonoid of finite index in a finitely generated inverse monoid is finitely generated as a closed inverse submonoid. 
\end{theorem}

\begin{proof}
Let $M$ be an inverse monoid generated by a set $X$.  We set $A = X \sqcup X^{-1}$: then each $s \in M$ can be expressed as a product $s=a_{1}a_{2} \dotsm a_{n}$
where $a_i \in A$. Suppose that $L$ is a closed inverse subsemigroup of finite index in $M.$  Let $\T$ be a transversal to $L$ in $M$. Given $h \in L,$ we write  $h=x_1 x_2 \dotsm x_n$ and consider the prefix $h_i = x_1 x_2 \dotsm x_i$ for
$1 \leq i \leq n$.  Since
\[ h_ih_i^{-1}hh^{-1}= h_ih_i^{-1}h_ix_{i+1} \dotsm x_n h^{-1}= h_ix_{i+1} \dotsm x_nh^{-1} = hh^{-1}, \]
we have $h_ih_i^{-1} \geq hh^{-1}$.  But $hh^{-1} \in L$ and $L$ is closed, so that $h_ih_i^{-1} \in L$.
Therefore the coset $\upset{Lh_i}$ exists, and so does the representative $\ol{h_i}$.  Now 
$$h = x_1x_2 \dotsm x_n \geq x_1 \cdot \ol{h_1}^{-1} \ol{h_1} \cdot x_{2} \cdot \ol{h_2}^{-1} \overline{h_2} \cdot  x_3  \cdot \ol{h_3}^{-1} \dotsm
\ol{h_{n-1}} \cdot  x_{n}\,.$$

By part (b) of Lemma \ref{delta} we have $\ol{h_j} = \ol{\ol{h_{j-1}}x_j}$ and so
$$h \geq x_1 \cdot \ol{x_1}^{-1} \ol{x_1} \cdot x_{2} \cdot (\ol{\ol{h_1}x_2})^{-1} \cdot \overline{h_2} \cdot  x_3  \cdot (\ol{\ol{h_2}x_3})^{-1} \dotsm
\ol{h_{n-1}} \cdot  x_{n} \,.$$ting that $1_M=\overline{x_1x_2 \dotsm x_n}$, we have
Now using the elements $\delta(r,s)$ from Lemma \ref{delta}, and no
$$h \geq \delta (1_M,x_{1}) \delta (\overline{x_1},x_2) \delta (\overline{h_2},x_3) \dotsm \delta(\ol{h_{n-1}},x_n) \,.$$
Finally, since $\upset{Lrs} = \upset{L\ol{rs}}$ then it follows from  Proposition
\ref{cosetsofL}(3) that $\delta(r,s) \in L$.  Hence $L$ is generated as a closed inverse submonoid of $M$ by the elements
$\delta(r,x)$ with $r \in \T$ and $x \in A$.
\end{proof}

\subsection{Recognizable closed inverse submonoids}
\begin{theorem}
\label{recog_iff_findex}
Let $L$ be a closed inverse submonoid of a finitely generated inverse monoid $M$.  Then the following are equivalent:
\begin{enumerate}
\item[(a)] $L$ is recognized by a
finite inverse automaton,
\item[(b)] $L$ has finite index in $M$,
\item[(c)] $L$ is a  recognizable subset of $M$.
\end{enumerate}
\end{theorem}

\begin{proof}
If $L$ has finite index in $M$, then by Lemma \ref{coset_aut}, its coset automaton $\C(M:L)$ is a finite inverse
automaton that recognizes $L$.  Conversely, suppose that $\A$ is a finite inverse automaton that recognizes $L$.  Again by Lemma \ref{coset_aut}, the coset automaton $\C$ is minimal, and so must be finite.  Hence (a) and (b) are equivalent.

If (b) holds, then as in the proof of Theorem \ref{Hall's theorem}, we obtain a homomorphism $M \ra \curlyI(D)$  for which $L$ is te invese image of the stabilizer of $1$.  Therefore
(b) implies (c).

We have $M$ generated by $X$, with $A = X \sqcup X^{-1}$, and a monoid homomorphism $\theta : A^* \ra M$.
To prove that (c) implies (b), suppose that $L$ is recognizable and so the language 
$\L=\{w\in A^*:  w \theta \in L\}$ is recognizable.  By Theorem \ref{nerode Thm}, the equivalence relation $ \simeq_{\L}$ on $A^*$  has finitely many classes. We claim that if $u \simeq_{\L} v$ and if $\upset{L(u \theta)}$ exists, then $\upset{L(v \theta)}$ exists and 
$\upset{L(u \theta)} = \upset{L(v \theta)}$.
Now $(u\theta)\,(u\theta)^{-1} = (uu^{-1})\theta  \in L$ and hence $uu^{-1} \in \L$. But by assumption $u \simeq_{\L} v,$ and so $vu^{-1} \in \L$,  which implies that $(v\theta)(u\theta)^{-1} \in L.$ By part (c) of Proposition \ref{cosetsofL}, $\upset{L(v \theta)}$ exists and
$\upset{L(v\theta)} = \upset{L(u\theta)}$.
Since $ \simeq_L$ has only finitely many classes, there are only finitely many cosets of $L$ in $M$.  Hence
(c) implies (b).
\end{proof}

\subsection{Rational Generation}
In this section we give an automata-based proof of part of \cite[Theorem 3.7]{MarMea}.  We adapt the
approach used in \cite[Theorem II]{FrouSakSch} to the proof of the following theorem of Anisimov and Seifert.

\begin{theorem}{\cite[Theorem 3]{AnisSeif}}
\label{anisseif}
A subgroup of a finitely generated group $G$ is a rational subset of $G$ if and only if it is finitely generated.
\end{theorem}

\begin{theorem}
\label{rational gen}
Let $L$ be a closed inverse submonoid  of a finitely generated inverse 
monoid $M$.
Then $L$ is generated as a closed inverse submonoid by a rational subset if and only if $L$ is 
generated as a closed inverse submonoid by a finite subset.
\end{theorem}

\begin{proof}
Since finite sets are rational sets, one half of the theorem is trivial. 

So suppose that $L$ is generated (as a closed inverse submonoid) by some rational subset $Y$ of $L$.  As above, if 
$M$ is generated (as an inverse monoid) by $X$, we take $A = X \cup X^{-1}$,
and  let $ \theta$ be the obvious map $A^* \ra M$.  
Then $Z = (Y \cup Y^{-1})^*$ is rational and so there exists a rational language
$R$ in $A^*$ such that $R  \theta = Z$.  The pumping lemma for $R$ then tells us that there exists a constant $C$ 
such that, if $w \in R$ with $|w| > C$, then
$w=uvz$ with $|uv| \leq C, |v| \geq 1$, and $uv^iz \in R$ for all $i \geq 0$. We set
\[ U = \{ uvu^{-1} : u,v \in A^*, |uv| \leq C, (uvu^{-1}) \theta\in L \} \]
and $V = \nobraupset{ \<U \theta \>}$.
Clearly $U$ is finite, and $V \subseteq L$.  We claim that $L=V$, and so we shall show that $R \theta \subseteq V$.

We first note that if $w \in R$ and $|w| \leq C$ then $w \in U$ (take $u=1, 
v=w$) and so $w \theta \in V$.  Now suppose that $|w|>C$ but that there exists $n \in L \setminus V$ with $n \geq w \theta$
Choose $|w|$ minimal.
The pumping lemma gives $w=uvz$ as above. Since $|uz| < |w|$
it follows that $(uz)  \theta \in V$.

Moreover,
\[ (uvu^{-1})  \theta \geq (uvzz^{-1}u^{-1}) \theta = (uvz)  \theta ((uz) 
\theta)^{-1} = (w  \theta)\,((uz) \theta)^{-1}. \]
Now $w  \theta \in L$ and $(uz)  \theta \in V$ : since $L$ is closed, 
$(uvu^{-1}) \theta \in L$ and therefore $uvu^{-1} \in U$.
Now
\[ n \geq w \theta = (uvz) \theta \geq (uvu^{-1}uz) \theta = (uvu^{-1})\theta\,\, 
(uz) \theta \in V . \]

Since $V$ is closed, $n \in V$.  But this is a contradiction.  Hence $L=V$.
\end{proof}

\begin{cor}
\label{rat_imp_fg}
If a closed inverse submonoid  $L$ of a finitely generated inverse 
monoid $M$ is a rational subset of $M$ then it is 
finitely generated as a closed inverse submonoid.
\end{cor}

\begin{proof}
If $L$ is a rational subset of $M$ then it is certainly generated by a rational set, namely $L$ itself.
\end{proof}

However, the converse of Corollary \ref{rat_imp_fg} is not true. We shall use the following Lemma to validate a counterexample
in Example \ref{f2ab}.

\begin{lemma}
\label{star-height}
Let $M$ be a semilattice of groups $G_1 \sqcup G_0$ over the semilattice $1>0$, and suppose that $T$ is a rational subset of
$M$ of star-height $h$.  Then $G_1 \cap T$ is a rational subset of $G_1$.
\end{lemma}

\begin{proof}
We proceed by induction on $h$.  If $h=0$ then $T$ is finite, and $G_1 \cap T$ is a finite subset of $G_1$ and so is a rational
subset of $G_1$, also of star-height $h_1=0$.

If $h>0,$ then, as in section \ref{fg_and_fi}, $T$ is a finite union $T = S_1 \cup \dotsb \cup S_k$
where each $S_j$ is a product $S_j = R_1 R_2 \dotsm R_{m_j}$ and where each $R_i$ is either a singleton subset of $M$ or $R_i = Q_i^*$ for some rational subset $Q_i$ of $M$ of star-height $h-1$ (see \cite[Chapter III]{Ber}).
Hence
\[ G_1 \cap T = (G_1 \cap S_1) \cup \dotsb \cup (G_1 \cap S_k) \,.\]
Consider the subset $G_1 \cap S_j = G_1 \cap R_1 R_2 \dotsm R_{m_j}$.  We claim that
\begin{equation} \label{rat_prod} G_1 \cap R_1 R_2 \dotsm R_{m_j} = (G_1 \cap R_1)(G_1 \cap R_2) \dotsm (G_1 \cap R_{m_j}) \,.  \end{equation}
The inclusion $\supseteq$ is clear, and so now we suppose that $g \in G_1$ is a product
$g = r_1 r_2 \dotsm r_{m_j}$ with $r_i \in R_i$.  If any $r_i \in G_0$ then $g \in G_0$: hence each $r_i \in G_1$ and so
$g \in (G_1 \cap R_1)(G_1 \cap R_2) \dotsm (G_1 \cap R_{m_j}) \,,$ confirming \eqref{rat_prod}.

The factors on the right of \eqref{rat_prod} are either singleton subsets of $G_1$, or are of the form $G_1 \cap Q_i^*$
where $Q_i$ is a rational subset of $M$ of star-height $h-1$.  However, $G_1 \cap Q_i^* = (G_1 \cap Q_i)^*$: the inclusion $G_1 \cap Q_i^* \supseteq (G_1 \cap Q_i)^*$ is again obvious,
and $G_1 \cap Q_i^* \subseteq (G_1 \cap Q_i)^*$ since if $w = x_1 \dots x_m \in Q_i^*$ and some $x_j \in G_0$ then
$w \in G_0$.  It follows that if $w \in G_1 \cap Q_i^*$ then $x_j \in G_1$ for all $j$.

Hence $G_1 \cap T$ is a union of  subsets, each of which is a product of singleton subsets of $G_1$ and subsets of the form $(G_1 \cap Q_i)^*$ where, by induction
$G_1 \cap Q_i$ is  a rational subset of $G_1$ of star-height $h_2 \leq h-1$.  Therefore $G_1 \cap T$ is a rational subset of $G_1$.
\end{proof}

\begin{cor}\label{rational_clifford}
Let $L = L_1 \sqcup L_0$, where $L_j \subseteq G_j$, be an inverse submonoid of $M$ that is also a rational subset of $M$.  Then $L_1$ is a rational subset of $G_1$.
\end{cor}

\begin{proof}
Take $T=L$: then $G_1 \cap L = L_1$ is a rational subset of $G_1$.
\end{proof}
Now, we show that the converse of Corollary \ref{rat_imp_fg} is not true.

\begin{example}
\label{f2ab}
Let $F_2$ be a free group of rank $2$ and consider the semilattice of groups  $M = F_2 \sqcup F_2^{ab}$ determined by the abelianisation map
$\alpha : F_2 \ra F_2^{ab}$.   The kernel of $\alpha$ is the commutator subgroup
$F_2'$ of $F_2$ and we let $K$ be the closed inverse submonoid $F_2' \sqcup \{ {\mathbf 0} \}$.
\[ \xymatrixcolsep{3pc}
\xymatrix{
F_2' \ar[d] \ar[r] & F_2 \ar[d]^{\alpha}\\
\{ {\mathbf 0} \} \ \ar[r] & F_2^{ab}}
\]
Now $K$ is generated (as a closed inverse submonoid) by $\{ {\mathbf 0} \}$ and so is finitely generated.  But $F_2'$ is not
finitely generated as a group (see \cite[Example III.4(4)]{BaumBook}) and so is not a rational subset of $F_2$ by
Theorem \ref{anisseif}.  Therefore, by Corollary \ref{rational_clifford}, $K$ is not a rational subset of $M$.

This example also gives us a counterexample to the converse of Theorem \ref{Schr's lemma}: $K$ is
finitely generated as a closed inverse submonoid, but has infinite index in $M$.
\end{example}

\section{Conclusion}
We summarize our findings about the conditions considered by Margolis and Meakin in \cite[Theorem 3.5]{MarMea}.

\begin{theorem}
\label{gen_MMthm}
Let $L$ be a closed inverse submonoid of the finitely generated inverse monoid $M$ and consider the following properties that $L$
might possess:
\begin{enumerate}
\item[(a)] $L$ is recognized by a finite inverse automaton,
\item[(b)] $L$ has finite index in $M$,
\item[(c)] $L$ is a  recognizable subset of $M$,
\item[(d)] $L$ is a rational subset of $M$,
\item[(e)] $L$ is finitely generated as a closed inverse submonoid of $M$..
\end{enumerate}
Then properties (a), (b) and (c) are equivalent: each of them implies (d), and (d) implies (e).  The latter two implications
are not reversible.
\end{theorem}

\begin{proof}
The equivalence of (a), (b) and (c) was established in Theorem \ref{recog_iff_findex}, and that (d) implies (e)
in Corollary \ref{rat_imp_fg}.  The implication  that (c) implies (d) is McKinight's Theorem \ref{mcknight}.

Counterexamples for (e) implies (d) and (e) implies (b) are given in Example \ref{f2ab}
\end{proof}

\end{document}